\documentclass[12pt]{amsart}
\usepackage{amssymb,latexsym,amsmath,amscd,amsthm,graphicx, color}
\usepackage[all]{xy}
\usepackage{pgf,tikz}
\usepackage{mathrsfs}
\usepackage{graphicx}
\usepackage{epstopdf}
\usepackage{breqn}
\usetikzlibrary{arrows}
\definecolor{uuuuuu}{rgb}{0.26666666666666666,0.26666666666666666,0.26666666666666666}
\definecolor{xdxdff}{rgb}{0.49019607843137253,0.49019607843137253,1.}
\definecolor{ffqqqq}{rgb}{1.,0.,0.}

\newtheorem{theorem}{Theorem}[section]
\newtheorem{lemma}[subsection]{Lemma}

\theoremstyle{definition}
\newtheorem{remark}[subsection]{Remark}
\newtheorem{definition}[subsection]{Definition}

\newtheorem{corollary}[theorem]{Corollary}
\theoremstyle{remark}

\numberwithin{equation}{section}

\title{Analysis of Fractal Dimension of Mixed Riemann-Liouville Fractional Integral}

\author{Subhash Chandra and Syed Abbas}
\address{School of Basic Sciences, Indian Institute of Technology Mandi\\ Kamand (H.P.) - 175005, India}
\email{sahusubhash77@gmail.com; sabbas.iitk@gmail.com (Email of corresponding author)}
\begin{document}

\begin{abstract}
 
In this article, we investigate fractal dimension of the graph of the mixed Riemann-Liouville fractional integral for various choice of continuous functions on a rectangular region. We estimate bounds for the box dimension and the Hausdorff dimension of the graph of the mixed Riemann-Liouville fractional integral of the functions which belong to the class of continuous functions and the class of H\"{o}lder continuous functions. We also show that the box dimension of the graph of the mixed Riemann-Liouville fractional integral of two-dimensional continuous functions is also two. Furthermore, we give construction of unbounded variational continuous functions. Later, we prove that the box dimension and the Hausdorff dimension of the graph of the mixed Riemann-Liouville fractional integral of unbounded variational continuous functions are also two.
\end{abstract}

\maketitle

%%%%%%%%%%%%%%%%%%%%%%%%%%%%%%%%%%%%%%%%%%%%%%%%%%%%%%%%%%%%%%%%%%%%%%%%

%%%%%%%%%%%%%%%%%%%%%%%%%%%%%%%%%%%%%%%%%%%%%%%%%%%%%%%%%%%%%%%%%%%%%%%

\section{\textbf{Introduction}}
Fractional calculus (FC) and fractal geometry (FG) have become rapidly growing fields in theory as well as applications. In the past, mathematics was primarily concerned with sets and functions on which classical calculus methods could be applied, and study of irregular and non-smooth sets or functions have been ignored. Although irregular sets are much better at representing certain natural phenomena than the figures of classical geometry do. FG provides a broad context for studying such irregular sets. Since the last few decades, several researchers have been fascinated by the graph of a function, its Hausdorff dimension, and box dimension. The study of dimensions of graphs began with Weierstrass type functions. Readers may encourage to see \cite{K,BR,W}, for the Hausdorff dimension and the box dimension of Weierstrass type functions. We refer the books \cite{MF1} and \cite{F} on FG, for more details. FC deals with the concept of non-integer order differentiation and integration and it is as old as classical calculus. Generally fractional derivatives are represented in terms of fractional integrals, in FC, for instance, we refer \cite{Samko,I}. Random fractals can be considered as better example of irregular functions and for analyzing such functions, FC is the best mathematical operator. Nowadays researchers are interested in the fractal dimension of graph of fractional integrals and derivatives. A connection between FC and fractal dimension can be seen in \cite{L1,L3,Ruan,FB,WU1,WU2,V2,Yao}. In the smoothness analysis of any irregular function, the box dimension plays an important role. Now, we will look over some of the available results on fractional calculus and fractal dimension. Liang \cite{L3} investigated the the box dimension of the graph of the fractional integral of Riemann-Liouville (R-L) type corresponding to a function having box dimension one. We know that  in the study of rectifiable curves and integrals, the bounded variation property of any function plays a significant role. An important result on box dimension of a function which is of bounded variation and continuous is given in \cite{L1}. In \cite{L1}, Liang proved that if $f\in C([0,1])$ and of bounded variation on $[0,1]$,  then $\dim_B Gr(f,[0,1])=1$, and $\dim_B Gr(\mathcal{I}^{\nu} f,[0,1])=1,$ where $$\mathcal{I}^\nu f=\frac{1}{\Gamma (\nu)}\int_0^x (x-s)^{\nu-1}f(s)ds,$$ is the fractional integral of R-L type. Now, we are interested in the notions of bounded variation for several variables and we will see that how these notions play an important role for  the study of fractal dimension of the graph of the fractional integral of mixed R-L type. Clarkson and Adams introduced the new notions of bounded variation such as Hahn, Peirpont and Arzel\'{a} in \cite{CK2} and related properties are given in \cite{CK1}. Using the bounded variation property in Arzel\'{a} sense, Verma and Viswanathan established the results for the fractional integral of mixed R-L type in \cite{V1}. Additionally, they proved that if $f\in C([a,b]\times [c,d])$  and $f$ is of bounded variation in sense of Arzel\'{a} on $[a,b]\times [c,d]$, then $\dim_B Gr(f,[a,b]\times [c,d])=2$, and $\dim_B Gr(\mathcal{I}^{\gamma} f,[a,b]\times [c,d])=2,$ where
$$ \mathcal{I}^{\gamma}f(x,y)=\frac{1}{\Gamma (\gamma_1)  \Gamma (\gamma_2)} \int_a ^x \int_c ^y (x-u)^{\gamma_1-1} (y-v)^{\gamma_2-1}f(u,v)dudv ,$$ with $\gamma = ( \gamma_1, \gamma_2 )$; $ \gamma_1 >0 , \gamma_2 >0,$ is the fractional integral of mixed R-L type. Although some examples can be found of two-dimensional continuous functions which are not of bounded variation in \cite{V1} and result on unbounded variation points for the fractional integral of R-L type can be found in \cite{V3}. Feng \cite{FZ} studied some properties of the variation and oscillation of  bivariate continuous functions. Also, he investigated Minkowski dimension of the fractal interpolation surface (FIS). Feng and Sun introduced a new construction method of FIS by considering arbitrary interpolation nodes in \cite{FX}, and they estimated the box dimension of FIS. We proved that the fractional integral of mixed R-L type of FIS is again FIS in \cite{SC}.\\
From  the above discussion, it is natural to arise the following questions:

\begin{itemize}
\item[(i)] What is the bounds of the box dimension and the Hausdorff dimension of the graph of $\mathcal{I}^\gamma f$  when $f\in C(I\times J),$ where $C(I\times J)$ denotes the set of all continuous functions on $I\times J.$
\item[(ii)] What is the bounds of the box dimension and the Hausdorff dimension of the graph of $\mathcal{I}^\gamma f$  when $f\in H^\mu(I\times J),$ where $H^\mu(I\times J)$ denotes the set of all H\"{o}lder continuous functions on $I\times J.$
\item[(iii)] What is the box dimension and the Hausdorff dimension of the graph of $\mathcal{I}^\gamma f$ when $f$ is unbounded variational continuous function.
\item[(iv)] What is the box dimension of the graph of $\mathcal{I}^\gamma f$ when $f$ is two-dimensional continuous function.
\end{itemize}
Above Questions (i),(ii) $\&$ (iii) are based on analytical aspects in the sense that we are using fundamental properties of function $f$. Question (iv) is based on dimensional aspects in the sense that we are using dimension of function $f$ to compute the dimension of the the graph of $\mathcal{I}^\gamma f$. 
%%%%%%%
%From above discussion, it is natural to arise the question that what is the bounds for the box dimension and the Hausdorff dimension of the graph of mixed Riemann-Liouville fractional integral of a function which belongs to class of continuous functions and class of H\"{o}lder continuous function on a rectangular region. Another question is that for which class of functions the fractal dimension of the graph of mixed Riemann-Liouville fractional integral is two. Here, we will investigate fractal dimension for every choice of functions ...questioned/mentioned above.......
%%%%%%
 In this work, we investigate the above mentioned points.
\\This article is arranged as follows:  Definitions of the mixed R-L fractional integral, box dimension, Hausdorff dimension and other basic terminologies are given in Section 2. In Sections 3 $\&$ 4, we provide bounds for the box dimension and the Hausdorff dimension of the graph of the fractional integral of mixed R-L type of various choice of functions. In Section 5, we estimate the box dimension of the graph of the fractional integral of mixed R-L type of a continuous function having box dimension two. Section 6 is devoted to the construction of unbounded variational continuous function and the fractal dimensions of its  fractional integral of mixed R-L type.
%Continue....in \cite{V1}.....but no one discussed about the bounds of %the Housdorff dimension and box dimension of the graph of the mixed %Riemann Liouville fractional integral corresponding to continuous %functions and H\"{o}lder continuous functions.

\section{\textbf{Preliminaries}}
Let us recall basic definitions and other terminologies which act as prelude to our article.

\textit{2.1. Mixed Riemann-Liouville fractional integral}
\begin{definition}\cite{Samko}\label{Def1}
 Let a function $f$ which is defined on a closed rectangle $[a,b] \times [c,d]$ and $a\geq 0,c\geq 0.$ Assuming that the following integral exists, mixed Riemann-Liouville fractional integral of $f$ is defined by $$ \mathcal{I}^{\gamma}f(x,y)=\frac{1}{\Gamma (\gamma_1)  \Gamma (\gamma_2)} \int_a ^x \int_c ^y (x-u)^{\gamma_1-1} (y-v)^{\gamma_2-1}f(u,v)dudv ,$$ where $\gamma = ( \gamma_1, \gamma_2 )$ with $ \gamma_1 >0 , \gamma_2 >0.$
\end{definition}
\textit{2.2. Fractal dimensions}\\
For the definition of the fractal dimensions, reader may follow \cite{F}.
%Let $F \neq\emptyset $ and subset of $\mathbb{R}^n$. The diameter of $F$ is given by $$\lvert F \rvert=\sup \left \{\lVert x-y \rVert_2:x,y\in F \right \}.$$ 
%If $\{F_i\}$ is a countable (or finite) collection of sets of diameter at most $\delta$ that cover $E\subseteq \mathbb{R}^n,$ then we say that  $\{F_i\}$ is a $\delta$-cover of $E.$ For $\delta>0$ and a non-negative real number $s$, we define $$H_\delta^s(E)=\inf \left \{\sum_{i=1}^\infty |F_i|^s: \{F_i\} ~\text{is a}~ \delta~-\text{cover of}~ E\right \}.$$
%\begin{definition} The $s$-dimensional Hausdorff measure of $E$ is defined as $H^s(E)=\lim_{\delta \to 0} H_\delta^s(E).$
%\end{definition}
%\begin{definition} Let $s\ge 0$ and $E\subseteq \mathbb{R}^n.$ The Hausdorff dimension of $E$ is defined as 
%$$\dim_H(E)=\inf  \{s:H^s(E)=0 \}=\sup \{s:H^s(E)=\infty\}.$$
%\end{definition}
\begin{definition} \cite{F}\label{DefB}
Let $E\neq \emptyset$ be a bounded subset of $\mathbb{R}^n$. Let the smallest number of sets which can cover $E$  is denoted by $N_{\delta}(E)$ having diameter at most $\delta$.  Then 
\begin{equation}
\underline{\dim}_B(E)=\mathop{\underline{\lim}}_{\delta \to 0}\frac{\log N_\delta(E)}{-\log\delta}~~~~~\text{(Lower box dimension)}
\end{equation}
and 
\begin{equation}
\overline{\dim}_B(E)=\overline{\lim_{\delta \to 0}}\frac{\log N_\delta(E)}{-\log\delta}~~~~~\text{(Upper box dimension)}.
\end{equation}
If $\underline{\dim}_B(E)=\overline{\dim}_B(E)$, the common value is called the box dimension of $E$. That is,
\begin{equation*}
\dim_B(E)=\lim_{\delta \to 0}\frac{\log N_\delta(E)}{-\log\delta}.
\end{equation*}
\end{definition}
\textit{2.3. Range of $f$}
\begin{definition}
For a function $f:A=:[a,b]\times [c,d]\to \mathbb{R}$, the maximum range of $f$ over $A$ is defined by $$ R_f[A]:=\sup_{(t_1,t_2),(x,y)\in A}\lvert f(t_1,t_2)-(x,y)\rvert.$$
\end{definition}
\begin{lemma}\cite{V1}\label{lm1}
Let $f \in C(I\times J)$ and
\begin{equation}\label{eq1}
\lvert f(z_1,t_1)-f(z_2,t_2) \rvert \leq C \lVert (z_1,t_1)-(z_2,t_2)\rVert_{2}^{\mu}, ~~~~\forall (z_1,t_1), (z_2,t_2)\in I\times J,
\end{equation}
for $C>0$ and $0\leq \mu \leq 1.$ Then $2\leq {\dim}_H Gr(f,I\times J)\leq\overline{\text{dim}}_B Gr(f,I\times J)\leq 3-\mu.$ This remains true if \ref{eq1}(H\"{o}lder condition) holds with $\lVert (z_1,t_1)-(z_2,t_2) \rVert_2<\delta$ for some $\delta>0.$\\
If $\mu=1$, then $f$ is called Lipschitz continuous.
\end{lemma}
 \begin{lemma}\label{lm2}
 For $0<\mu <1$ and $C>0$, let
\begin{equation*}
H^{\mu}(I \times J)=\{f(x,y):\lvert f(x+k_1,y+k_2)-f(x,y)\rvert \leq C \lVert (k_1,k_2)\rVert_{2}^{\mu},~\forall ~(x+k_1,y+k_2),(x,y)\in I\times J \} .
\end{equation*}
If  $f \in C(I\times J)$ and belongs to $H^\mu(I \times J)$, then
$$2\le {\dim}_H Gr(f,I\times J)\leq\overline{\dim}_B Gr(f,I\times J)\leq 3-\mu.$$
 \end{lemma}
 Reader may refer \cite{CK1} for the definition of bounded variation in Arzel\'{a} sense.
\begin{theorem}\cite{CK1} \label{th2}(Necessary and sufficient condition)
 \\A function $g:[a,b]\times [c,d]\to \mathbb{R}$ is said to be of bounded variation in the sense of Arzel\'{a} if it  can be written in the difference of two bounded functions $g_1$ and $g_2$  satisfying the inequities
$$\Delta_{10}g_i(x,y)\ge 0,~~\Delta_{01}g_i(x,y)\ge 0,~~i=1,2,$$
where $\Delta_{10}g(x_i,y_j)=g(x_{i+1},y_j)-g(x_i,y_j),~\Delta_{01}g(x_i,y_j)=g(x_i,y_{j+1})-g(x_i,y_j).$
\end{theorem}
 Following notations are also used in this article: $Gr(f)$ represents the graph of $f$. $I \times J =[a,b]\times [c,d]$. $C$ is absolute constant and  it may have different values even in the same line at different occurrence. Sometimes, we use the abbreviation ``the fractional integral of mixed R-L type" in the place of  ``the mixed Riemann-Liouville fractional integral". 
%%%%%%%%%%%%%%%%%%%%%%%%%%%%%%%%%%%%%%%%%%%%%%%%%%%%%%%%%%%
\section{\textbf{Fractal Dimensions of $\mathcal{I}^{\gamma}f(x,y)$ with $f(x,y)\in C(I\times J )$}}
In this section, we establish the bounds for the fractal dimension of the fractional integral of mixed R-L type corresponding to a continuous function. 

\begin{theorem}\label{Th1}
For $0<a<b<\infty,~0<c<d<\infty$ and $0<\gamma_1,\gamma_2<1$. If $f:[a,b]\times [c,d] \to \mathbb{R}$ is continuous, then $${\dim}_H Gr(\mathcal{I}^{\gamma}f,I\times J)\leq\overline{\dim}_B Gr(\mathcal{I}^{\gamma}f,I\times J)\leq 3-{\min}\{\gamma_1,\gamma_2\}.$$
\end{theorem}

\begin{proof}
Let $0<a\leq x<x+k_1\leq b;~0<c\leq y<y+k_2\leq d.$ Then 
\begin{dmath*}
(\mathcal{I}^{\gamma}f)(x+k_1,y+k_2)-(\mathcal{I}^{\gamma}f)(x,y)\\
=\frac{1}{\Gamma (\gamma_1)  \Gamma (\gamma_2)} \int_a ^{x+k_1} \int_c ^{y+k_2} (x+k_1-u)^{\gamma_1-1} (y+k_2-v)^{\gamma_2-1}f(u,v)dudv\\-\frac{1}{\Gamma (\gamma_1)  \Gamma (\gamma_2)} \int_a ^{x} \int_c ^{y} (x-u)^{\gamma_1-1} (y-v)^{\gamma_2-1}f(u,v)dudv
=L_1+L_2+L_3+L_4,\\
\end{dmath*}
where
\begin{equation*}
\begin{aligned}
L_1=&\frac{1}{\Gamma (\gamma_1)  \Gamma (\gamma_2)} \int_a ^x \int_c ^y \left[ (x+k_1-u)^{\gamma_1-1} (y+k_2-v)^{\gamma_2-1}- (x-u)^{\gamma_1-1} (y-v)^{\gamma_2-1}\right] f(u,v) dudv\\
L_2=&\frac{1}{\Gamma (\gamma_1)  \Gamma (\gamma_2)} \int_a ^x \int_y ^{y+k_2} (x+k_1-u)^{\gamma_1-1} (y+k_2-v)^{\gamma_2-1}f(u,v)dudv\\
L_3=&\frac{1}{\Gamma (\gamma_1)  \Gamma (\gamma_2)} \int_x ^{x+k_1} \int_c ^{y} (x+k_1-u)^{\gamma_1-1} (y+k_2-v)^{\gamma_2-1}f(u,v)dudv\\
L_4=&\frac{1}{\Gamma (\gamma_1)  \Gamma (\gamma_2)} \int_x ^{x+k_1} \int_y ^{y+k_2} (x+k_1-u)^{\gamma_1-1} (y+k_2-v)^{\gamma_2-1}f(u,v)dudv.
\end{aligned}
\end{equation*}
Because of continuity of $f$ on $[a,b]\times [c,d]$, there exists $M$ such that 
$\lvert f(t_1,t_2)\rvert \leq M~\forall (t_1,t_2)\in [a,b]\times [c,d].$\\
Now, we estimate the bound for $L_1$ as bellow:
\begin{equation*}
\begin{aligned}
\lvert L_1\rvert\leq &\frac{1}{\Gamma (\gamma_1)  \Gamma (\gamma_2)} \int_a ^x \int_c ^y \left[ (x-u)^{\gamma_1-1} (y-v)^{\gamma_2-1}-(x+k_1-u)^{\gamma_1-1} (y+k_2-v)^{\gamma_2-1}\right] \lvert f(u,v) \rvert dudv\\
\leq &\frac{M}{\Gamma (\gamma_1)  \Gamma (\gamma_2)} \int_a ^x \int_c ^y \left[(x-u)^{\gamma_1-1} (y-v)^{\gamma_2-1}-(x+k_1-u)^{\gamma_1-1} (y+k_2-v)^{\gamma_2-1}\right] dudv\\
\end{aligned}
\end{equation*}
\begin{dmath*}
 =\frac{M}{\Gamma (\gamma_1)  \Gamma (\gamma_2)} \int_a ^x \int_c ^y \left[(x-u)^{\gamma_1-1} (y-v)^{\gamma_2-1}-(x+k_1-u)^{\gamma_1-1} (y-v)^{\gamma_2-1}\\
 +(x+k_1-u)^{\gamma_1-1} (y-v)^{\gamma_2-1}-(x+k_1-u)^{\gamma_1-1} (y+k_2-v)^{\gamma_2-1}\right] dudv
  =\frac{M}{\Gamma (\gamma_1)  \Gamma (\gamma_2)} \left[ \int_a ^x \int_c ^y (y-v)^{\gamma_2-1} \left[(x-u)^{\gamma_1-1}-(x+k_1-u)^{\gamma_1-1}\right] dudv\\
 + \int_a ^x \int_c ^y (x+k_1-u)^{\gamma_1-1} \left[ (y-v)^{\gamma_2-1}- (y+k_2-v)^{\gamma_2-1}\right] dudv \right].
\end{dmath*}
Let $J_1$ and $J_2$ defined as follows and by using Bernoulli's inequality $(1+u)^{r'}\leq 1+r'u$ for $0\leq r' \leq 1$ and $u\geq -1$, we obtain
\begin{dmath*}
 J_1= \int_a ^x \left[(x-u)^{\gamma_1-1}-(x+k_1-u)^{\gamma_1-1}\right] du 
 =\frac{1}{\gamma_1} \left[ (x+k_1-x)^{\gamma_1}-(x+k_1-a)^{\gamma_1}+(x-a)^{\gamma_1}\right]
 =\frac{1}{\gamma_1} \left[ (k_1^{\gamma_1}-(x+k_1-a)^{\gamma_1}+(x-a)^{\gamma_1}\right]
\leq \frac{k_1^{\gamma_1}}{\gamma_1}.
\end{dmath*}
\begin{dmath*}
 J_2= \int_c ^y \left[ (y-v)^{\gamma_2-1}- (y+k_2-v)^{\gamma_2-1}\right] dv 
 =\frac{1}{\gamma_2} \left[ (y+k_2-y)^{\gamma_2}-(y+k_2-c)^{\gamma_2}+(y-c)^{\gamma_2}\right]
 =\frac{1}{\gamma_2} \left[ (k_2^{\gamma_2}-(y+k_2-c)^{\gamma_2}+(y-c)^{\gamma_2}\right]
\leq \frac{k_2^{\gamma_2}}{\gamma_2}.
\end{dmath*}
By using the values of $J_1$ and $J_2$, we get
\begin{dmath*}
\lvert L_1 \rvert \leq \frac{M}{\Gamma (\gamma_1)  \Gamma (\gamma_2)}\left[\frac{k_1^{\gamma_1}}{\gamma_1} \int_c^y (y-v)^{\gamma_2-1}dv+\frac{k_2}{\gamma_2}\int_a^x(x+k_1-u)^{\gamma_1-1}du\right]
\leq \frac{M}{\Gamma (\gamma_1)  \Gamma (\gamma_2)}\left[\frac{k_1^{\gamma_1}}{\gamma_1 \gamma_2}  (d-c)^{\gamma_2}+\frac{k_2^{\gamma_2}}{\gamma_1 \gamma_2}(b-a)^{\gamma_1}\right].
\end{dmath*}
Therefore for a suitable constant $C$, we obtain
\begin{dmath*}
\lvert L_1 \rvert \leq C(k_1^{\gamma_1}+k_2^{\gamma_2}).
\end{dmath*}
Now, we estimate $L_2$ as follows:
\begin{dmath*}
\lvert L_2 \rvert \leq \frac{1}{\Gamma (\gamma_1)  \Gamma (\gamma_2)} \int_a ^x \int_y ^{y+k_2} (x+k_1-u)^{\gamma_1-1} (y+k_2-v)^{\gamma_2-1}\lvert f(u,v)\rvert dudv
\leq \frac{M}{\Gamma (\gamma_1)  \Gamma (\gamma_2)} \int_a ^x \int_y ^{y+k_2} (x+k_1-u)^{\gamma_1-1} (y+k_2-v)^{\gamma_2-1} dudv
\leq \frac{(b-a)^{\gamma_1} k_2^{\gamma_2}}{\gamma_1 \gamma_2}.
\end{dmath*}
For suitable $C$, we get
\begin{dmath*}
\lvert L_2 \rvert \leq Ck_2^{\gamma_2}.
\end{dmath*}
Similarly \begin{dmath*}
\lvert L_3 \rvert \leq Ck_1^{\gamma_1}.
\end{dmath*}
In similar way, we estimate $L_4$
\begin{dmath*}
\lvert L_4 \rvert \leq \frac{1}{\Gamma (\gamma_1)  \Gamma (\gamma_2)} \int_x ^{x+k_1} \int_y ^{y+k_2} (x+k_1-u)^{\gamma_1-1} (y+k_2-v)^{\gamma_2-1}\lvert f(u,v) \rvert dudv
\leq \frac{M}{\Gamma (\gamma_1)  \Gamma (\gamma_2)} \int_x ^{x+k_1} \int_y ^{y+k_2} (x+k_1-u)^{\gamma_1-1} (y+k_2-v)^{\gamma_2-1} dudv
= \frac{k_1^{\gamma_1}k_2^{\gamma_2}}{\gamma_1 \gamma_2}.
\end{dmath*}
For suitable $C$, we have
\begin{dmath*}
\lvert L_4 \rvert \leq Ck_1^{\gamma_1}k_2^{\gamma_2}.
\end{dmath*}
Say $\alpha=\min\{\gamma_1,\gamma_2\}.$ For suitable $C$ and sufficiently small positive constants $k_1,k_2,\alpha$, we get
\begin{dmath*}
\lvert (\mathcal{I}^{\gamma}f)(x+k_1,y+k_2)-(\mathcal{I}^{\gamma}f)(x,y)\rvert \leq \lvert L_1 \rvert+\lvert L_2 \rvert+\lvert L_3 \rvert+\lvert L_4 \rvert
\leq C(k_1^{\gamma_1}+k_2^{\gamma_2})
\leq C(k_1^\alpha+k_2^\alpha).
\end{dmath*}
Since $k_1$ and $k_2$ are sufficiently small, we have $k_1\leq \sqrt{k_1^2+k_2^2}$ and $k_2\leq \sqrt{k_1^2+k_2^2}.$\\
Consequently, we get
\begin{dmath*}
\lvert (\mathcal{I}^{\gamma}f)(x+k_1,y+k_2)-(\mathcal{I}^{\gamma}f)(x,y)\rvert \leq  C \lVert (x+k_1,y+k_2)-(x,y)\rVert_2^\alpha.
\end{dmath*}
The proof follows from Lemma \ref{lm1}.
\end{proof}
\textbf{Semigroup property:}
\begin{theorem}\label{Th2}
Let $\gamma_1>0,\gamma_1'>0,\gamma_2>0,\gamma_2'>0$ and $0<a<b<\infty,0<c<d<\infty.$ Let $f:[a,b]\times [c,d] \to \mathbb{R}$ is an integrable function for which the fractional integral of mixed R-L type $\mathcal{I}^{(\gamma_1,\gamma_2)}f$ exists, then
\begin{dmath*}
\mathcal{I}^{(\gamma_1,\gamma_2)}\mathcal{I}^{(\gamma_1',\gamma_2')}f=\mathcal{I}^{(\gamma_1+\gamma_1',\gamma_2+\gamma_2')}f.
\end{dmath*}
\end{theorem}
\begin{proof}
From the Dirichlet technique and Fubini's theorem , we have
\begin{dmath*}
(\mathcal{I}^{(\gamma_1,\gamma_2)}\mathcal{I}^{(\gamma_1',\gamma_2')}f)(x,y)=\frac{1}{\Gamma (\gamma_1)  \Gamma (\gamma_2)\Gamma (\gamma_1')\Gamma (\gamma_2')} \int_a^x \int_c^y \left[ \int_s^x \int_t^y (x-v)^{\gamma_1-1}(v-s)^{\gamma_1'-1}\\.(y-w)^{\gamma_2-1}(w-t)^{\gamma_2'-1} dvdw\right]f(s,t)dsdt
\end{dmath*}
With the change of variable $z=\frac{v-s}{x-v}$, we have
\begin{dmath*}
\int_s^x (x-v)^{\gamma_1-1}(v-s)^{\gamma_1'-1}dv=(x-s)^{\gamma_1+\gamma_1'-1}\int_0^1 (1-z)^{\gamma_1-1}z^{\gamma_1'-1}dz
=(x-s)^{\gamma_1+\gamma_1'-1}\frac{\Gamma(\gamma_1)\Gamma (\gamma_1')}{\Gamma(\gamma_1+\gamma_1')},
\end{dmath*}
according to the known formulae for the beta function \cite{AA,I}.\\
Similarly
\begin{dmath*}
\int_t^y (y-w)^{\gamma_2-1}(w-t)^{\gamma_2'-1}dw=(y-t)^{\gamma_2+\gamma_2'-1}\int_0^1 (1-z)^{\gamma_2-1}z^{\gamma_2'-1}dz
=(y-t)^{\gamma_2+\gamma_2'-1}\frac{\Gamma(\gamma_2)\Gamma (\gamma_2')}{\Gamma(\gamma_2+\gamma_2')}.
\end{dmath*}
Consequently, we get
\begin{dmath*}
(\mathcal{I}^{(\gamma_1,\gamma_2)}\mathcal{I}^{(\gamma_1',\gamma_2')}f)(x,y)=\frac{1}{\Gamma(\gamma_1+\gamma_1')\Gamma(\gamma_2+\gamma_2')}\int_a^x \int_c^y (x-s)^{\gamma_1+\gamma_1'-1}(y-t)^{\gamma_2+\gamma_2'-1} f(s,t)dsdt
=(\mathcal{I}^{(\gamma_1+\gamma_1',\gamma_2+\gamma_2')}f)(x,y).
\end{dmath*}
Hence completes the proof.
\end{proof}

\begin{theorem}
Let $f:[a,b]\times[c,d]\to \mathbb{R}$ is continuous and $0<a<b<\infty,0<c<d<\infty.$
\begin{itemize}
\item[(1)] If $0<\gamma_1,\gamma_2<1$, then
$$2\leq {\dim}_H Gr(\mathcal{I}^{\gamma}f,I\times J)\leq \dim_B Gr(\mathcal{I}^{\gamma}f,I\times J)\leq 3-{\min}\{\gamma_1,\gamma_2\}.$$
\item[(2)] If $\gamma_1,\gamma_2 \geq 1$, then
$${\dim}_H Gr(\mathcal{I}^{\gamma}f,I\times J)={\dim}_B Gr(\mathcal{I}^{\gamma}f,I\times J)=2.$$
\end{itemize}
\end{theorem} 
The proof of the above theorem follows from Theorem \ref{Th1}, Theorem \ref{Th2} and from the relation between fractal dimensions.
%%%%%%%%%%%%%%%%%%%%%%%%%%%%%%%%%%%%%%%%%%%%%%%%%%%%%%%%%%%%%%%%%%%%%%%%%%%%%%%%5
\section{\textbf{Fractal Dimensions of $\mathcal{I}^{\gamma}f(x,y)$  with $f(x,y)\in H^{\mu}(I \times J)$}}
  In this section, we establish the bounds for the fractal dimension of the fractional integral of mixed R-L type corresponding to a $\mu$-H\"{o}lder continuous function.
\begin{theorem}\label{th4.1}
Let $f(x,y)\in H^{\mu}(I \times J)$ on $[a,b]\times[c,d]$ such that $f(0,0)=(0,0)$ and provided that the fractional integral of mixed R-L type of $f$ exists.
Then $${\dim}_H Gr(\mathcal{I}^{\gamma}f,I\times J)\leq\overline{\dim}_B Gr(\mathcal{I}^{\gamma}f,I\times J)\leq 3-\mu,~~~0<\gamma_1,\gamma_2<1.$$
\end{theorem}
\begin{proof}
Let $0\leq a \leq x<x+k_1\leq b$, $0\leq c \leq y<y+k_2\leq d$ and $0<\gamma_1,\gamma_2<1.$
Then
\begin{dmath*}
(\mathcal{I}^{\gamma}f)(x+k_1,y+k_2)-(\mathcal{I}^{\gamma}f)(x,y)\\
=\frac{1}{\Gamma (\gamma_1)  \Gamma (\gamma_2)} \int_a ^{x+k_1} \int_c ^{y+k_2} (x+k_1-u)^{\gamma_1-1} (y+k_2-v)^{\gamma_2-1}f(u,v)dudv\\-\frac{1}{\Gamma (\gamma_1)  \Gamma (\gamma_2)} \int_a ^{x} \int_c ^{y} (x-u)^{\gamma_1-1} (y-v)^{\gamma_2-1}f(u,v)dudv
=I_1+I_2+I_3+I_4-I_5,\\
\end{dmath*}
where
\begin{equation*}
\begin{aligned}
I_1=&\frac{1}{\Gamma (\gamma_1)  \Gamma (\gamma_2)} \int_a ^{a+k_1} \int_c ^{c+k_2} (x+k_1-u)^{\gamma_1-1} (y+k_2-v)^{\gamma_2-1}f(u,v)dudv\\
I_2=&\frac{1}{\Gamma (\gamma_1)  \Gamma (\gamma_2)} \int_a ^{a+k_1} \int_{c+k_2} ^{y+k_2} (x+k_1-u)^{\gamma_1-1} (y+k_2-v)^{\gamma_2-1}f(u,v)dudv\\
I_3=&\frac{1}{\Gamma (\gamma_1)  \Gamma (\gamma_2)} \int_{a+k_1} ^{x+k_1} \int_c ^{c+k_2} (x+k_1-u)^{\gamma_1-1} (y+k_2-v)^{\gamma_2-1}f(u,v)dudv\\
I_4=&\frac{1}{\Gamma (\gamma_1)  \Gamma (\gamma_2)} \int_{a+k_1} ^{x+k_1} \int_{c+k_2} ^{y+k_2} (x+k_1-u)^{\gamma_1-1} (y+k_2-v)^{\gamma_2-1}f(u,v)dudv\\
I_5=&\frac{1}{\Gamma (\gamma_1)  \Gamma (\gamma_2)} \int_a ^x \int_c ^y (x-u)^{\gamma_1-1} (y-v)^{\gamma_2-1}f(u,v)dudv.
\end{aligned}
\end{equation*}
By change of variable in $I_5$, we have
\begin{equation*}
\begin{aligned}
I'_5=&\frac{1}{\Gamma (\gamma_1)  \Gamma (\gamma_2)} \int_{a+k_1} ^{x+k_1} \int_{c+k_2} ^{y+k_2} (x+k_1-u)^{\gamma_1-1} (y+k_2-v)^{\gamma_2-1}f(u-k_1,v-k_2)dudv\\
I_4-I'_5=I_6=&\frac{1}{\Gamma (\gamma_1)  \Gamma (\gamma_2)} \int_{a+k_1} ^{x+k_1} \int_{c+k_2} ^{y+k_2} (x+k_1-u)^{\gamma_1-1} (y+k_2-v)^{\gamma_2-1}[f(u-k_1,v-k_2)-f(u,v)]dudv\\
\lvert I_6 \rvert\leq &\frac{1}{\Gamma (\gamma_1)  \Gamma (\gamma_2)} \int_{a+k_1} ^{x+k_1} \int_{c+k_2} ^{y+k_2}\lvert (x+k_1-u)^{\gamma_1-1} (y+k_2-v)^{\gamma_2-1}[f(u-k_1,v-k_2)-f(u,v)]\rvert dudv.
\end{aligned}
\end{equation*}
Since $f(x,y)\in H^{\mu}(I \times J)$ on $[a,b]\times[c,d]$, we have
\begin{equation*}
\begin{aligned}
\lvert I_6 \rvert\leq &\frac{C\lVert k_1,k_2\rVert^\mu_2}{\Gamma (\gamma_1)  \Gamma (\gamma_2)} \int_{a+k_1} ^{x+k_1} \int_{c+k_2} ^{y+k_2}\lvert (x+k_1-u)^{\gamma_1-1} (y+k_2-v)^{\gamma_2-1}\rvert dudv\\
=&\frac{C\lVert k_1,k_2\rVert^\mu_2}{\Gamma (\gamma_1+1)  \Gamma (\gamma_2+1)}(x-a)^{\gamma_1}(y-c)^{\gamma_2}
\end{aligned}
\end{equation*}
For $(x,y)\in [a,b]\times [c,d]$, we get
\begin{equation*}
\begin{aligned}
\lvert I_6 \rvert\leq &\frac{C\lVert k_1,k_2\rVert^\mu_2}{\Gamma (\gamma_1+1)  \Gamma (\gamma_2+1)}(b-a)^{\gamma_1}(d-c)^{\gamma_2}.\\
\lvert I_6 \rvert\leq &C\lVert k_1,k_2\rVert^\mu_2, ~~\text{where}~ C=\frac{(b-a)^{\gamma_1}(d-c)^{\gamma_2}}{\Gamma (\gamma_1+1)  \Gamma (\gamma_2+1)}.
\end{aligned}
\end{equation*}
Now for the bound of $I_1$, we apply similar steps as done above.
\begin{equation*}
\begin{aligned}
\lvert I_1\rvert \leq &\frac{1}{\Gamma (\gamma_1)  \Gamma (\gamma_2)} \int_a ^{a+k_1} \int_c ^{c+k_2}\lvert (x+k_1-u)^{\gamma_1-1} (y+k_2-v)^{\gamma_2-1}\rvert \lvert f(u,v)-f(0,0)\rvert dudv\\
\leq & \frac{C\lVert k_1,k_2\rVert^\mu_2}{\Gamma (\gamma_1)  \Gamma (\gamma_2)} \int_a ^{a+k_1} \int_c ^{c+k_2}\lvert (x+k_1-u)^{\gamma_1-1} (y+k_2-v)^{\gamma_2-1}\rvert dudv\\
\leq & \frac{C\lVert k_1,k_2\rVert^\mu_2}{\Gamma (\gamma_1)  \Gamma (\gamma_2)} \int_a ^{a+k_1} \int_c ^{c+k_2}\lvert (a+k_1-u)^{\gamma_1-1} (c+k_2-v)^{\gamma_2-1}\rvert dudv\\
=& \frac{C\lVert k_1,k_2\rVert^\mu_2}{\Gamma (\gamma_1+1)  \Gamma (\gamma_2+1)}k_1^{\gamma_1}k_2^{\gamma_2}.
\end{aligned}
\end{equation*}
So, we have
\begin{equation*}
\begin{aligned}
\lvert I_1 \rvert \leq & C\lVert k_1,k_2\rVert^\mu_2,~~\text{where}~C=\frac{k_1^{\gamma_1}k_2^{\gamma_2}}{\Gamma (\gamma_1+1)  \Gamma (\gamma_2+1)}.
\end{aligned}
\end{equation*}
In similar way, we obtain the bounds for $I_2$ and $I_3$ as follows
\begin{equation*}
\begin{aligned}
\lvert I_2 \rvert \leq & C\lVert k_1,k_2\rVert^\mu_2,~~\text{where}~C=\frac{k_1^{\gamma_1}(d-c)^{\gamma_2}}{\Gamma (\gamma_1+1)  \Gamma (\gamma_2+1)},\\
\lvert I_3 \rvert \leq & C\lVert k_1,k_2\rVert^\mu_2,~~\text{where}~C=\frac{(b-a)^{\gamma_1}k_2^{\gamma_2}}{\Gamma (\gamma_1+1)  \Gamma (\gamma_2+1)}.
\end{aligned}
\end{equation*}
Consequently, we get for a suitable constant $C$
\begin{equation*}
\begin{aligned}
\lvert (\mathcal{I}^{\gamma}f)(x+k_1,y+k_2)-(\mathcal{I}^{\gamma}f)(x,y)\rvert \leq  &\lvert I_1 \rvert+\lvert I_2 \rvert+\lvert I_3 \rvert+\lvert I_4 \rvert+\lvert I_5 \rvert\\
\leq & C\lVert k_1,k_2\rVert^\mu_2.
\end{aligned}
\end{equation*}
In view of Lemma \ref{lm2} the proof follows.
\end{proof}
\begin{remark}
 If $f(x,y)$ is any fractal function having box dimension $3-\mu$, then upper box dimension of the fractional integral of mixed R-L type corresponding to $f(x,y)$ is non-increasing.
Since, $${\dim}_B Gr(f,I\times J)= 3-\mu.$$
We have $$\overline{\dim}_B Gr(\mathcal{I}^{\gamma}f,I\times J)\leq 3-\mu.$$
That is $$\overline{\dim}_B Gr(\mathcal{I}^{\gamma}f,I\times J)\leq {\dim}_B Gr(f,I\times J)= 3-\mu.$$
\end{remark}
\begin{theorem}\label{4.2}
Let $f(x,y)$ be a continuous function defined on $[a,b]\times[c,d]$ with $f(0,0)=(0,0)$ and satisfies Lipschitz condition, then for $0<\gamma_1,\gamma_2<1$,
 \begin{equation*}
 {\dim}_H Gr(\mathcal{I}^{\gamma}f,I\times J)={\dim}_B Gr(\mathcal{I}^{\gamma}f,I\times J)=2.
 \end{equation*}
\end{theorem}
In view of Lemma \ref{lm1} and Theorem \ref{th4.1} the proof of the Theorem \ref{4.2} follows.
%%%%%%%%%%%%%%%%%%%%%%%%%%%%%%%%%%%%%%%%%%%%%%%%%%%%%%%%%%%%%%%%%%%%%
\section{\textbf{Fractal Dimension  of $\mathcal{I}^{\gamma}f(x,y)$ of 2-Dimensional Continuous Functions}}
 First we give the following Lemma \ref{lmm} which act as prelude for the  main Theorem \ref{Th5} and then we corroborate our results with help of existing results.
\begin{lemma}\label{lmm}
Let $f:[0,1]\times [0,1] \to \mathbb{R}$ is continuous and $0<\delta<1,~ \frac{1}{\delta} <m,n<1+\frac{1}{\delta}$ for some $m,n \in \mathbb{N}.$ If the number of $\delta$-cubes that intersect the graph $Gr(f)$ is denoted by $N_{\delta}(Gr(f))$, then
\begin{equation*}
\sum_{j=1}^n \sum_{i=1}^m \max \left \{\frac{R_f[A_{ij}]}{\delta},1 \right \} \leq N_\delta(Gr(f))\leq 2mn+\frac{1}{\delta} \sum_{j=1}^n \sum_{i=1}^m R_f[A_{ij}],
\end{equation*}
where $A_{ij}$ is the $(i,j)$-th cell corresponding to the net under consideration.
\end{lemma}
\begin{proof}
If $f(x,y)$ is continuous on $I\times J$, the number of cubes having side $\delta$ in the part above $A_{ij}$ which intersect $Gr(f,I\times J)$ is atleast $$\max \left \{\frac{R_f[A_{ij}]}{\delta},1 \right \}$$
and at most $$2+\frac{R_f[A_{ij}]}{\delta}.$$
By summing over all such parts we get the required result.
\end{proof}
\begin{theorem}\label{Th5}
Let a non-negative function $f(x,y)\in C([0,1]\times[0,1])$ and $0<\gamma_1<1,~0<\gamma_2<1$\\ If
 \begin{equation}\label{11}
\dim_B Gr(f, [0,1]\times[0,1])=2,
\end{equation}
then, the box dimension of the fractional integral of mixed R-L type of $f(x,y)$ of order $\gamma=(\gamma_1,\gamma_2)$ exists and is equal to $2$ on $[0,1]\times[0,1]$, as
\begin{equation}\label{eq12}
\dim_B Gr(\mathcal{I}^{\gamma}f,[0,1]\times[0,1])=2.
\end{equation}
\end{theorem}
\begin{proof}
Since $f(x,y)\in C([0,1]\times[0,1])$, $\mathcal{I}^{\gamma}f(x,y)$ is also continuous on $[0,1]\times[0,1]$ (from Theorem 4.2 in \cite{V1}). From the definition of the box dimension, we can get 
\begin{equation}\label{12}
\underline{\dim}_B Gr(\mathcal{I}^{\gamma}f,[0,1]\times[0,1]) \geq 2.
\end{equation}
To prove Equation \ref{eq12}, we have to prove the following inequality
\begin{equation}\label{13}
\overline{\dim}_B Gr(\mathcal{I}^{\gamma}f,[0,1]\times[0,1]) \leq 2.
\end{equation}
Suppose that $0<\delta<\frac{1}{2}$, $ \frac{1}{\delta} <m,n<1+\frac{1}{\delta}$ and  $N_{\delta}(Gr(f))$ is the number of $\delta$-cubes that intersect $Gr(f)$. From Equation \ref{11}, it holds
\begin{equation*}
\lim_{\delta \to 0} \frac{\log N_{\delta}(Gr(f))}{-\log\delta}=2.
\end{equation*}
Let $N_{\delta}(Gr(\mathcal{I}^\gamma f))$ is the number of $\delta$-cubes that intersect $Gr(\mathcal{I}^\gamma f)$. Thus Equation \ref{13} can be written as
\begin{equation}\label{14}
\overline{\lim_{\delta \to 0} } \frac{\log N_{\delta}(Gr(\mathcal{I}^\gamma f))}{-\log\delta} \leq 2.
\end{equation}
Now, we are ready to prove Equation \ref{14}.\\ 
Let $f(x,y)\in C([0,1]\times [0,1])$ and $0<\gamma_1, \gamma_2<1$. If $k_1>0,k_2>0~ \text{and}~x+k_1\leq 1, y+k_2\leq 1,$ then
\begin{dmath*}
(\Gamma (\gamma_1)  \Gamma (\gamma_2))\left[(\mathcal{I}^\gamma f)(x+k_1,y+k_2)-(\mathcal{I}^\gamma f)(x,y)\right]\\
=\int_0^{x+k_1} \int_0^{y+k_2}(x+k_1-u)^{\gamma_1-1}(y+k_2-v)^{\gamma_2-1}f(u,v)dudv\\-\int_0^{x} \int_0^{y}(x-u)^{\gamma_1-1}(y-v)^{\gamma_2-1}f(u,v)dudv.
\end{dmath*}
By integral transform, let $$\left(\frac{u}{x+k_1}\right)=s,$$ and  $$\left(\frac{v}{y+k_2}\right)=t.$$ Then $$dudv =|J|dsdt,$$ where
\[ 
J=
\begin{bmatrix}
\frac{\partial u}{\partial s} & \frac{\partial u}{\partial t}\\
\frac{\partial v}{\partial s} & \frac{\partial v}{\partial t}
\end{bmatrix}
\]
$$=(x+k_1)(y+k_2).$$
Thus, we have
\begin{dmath*}
(\Gamma (\gamma_1)  \Gamma (\gamma_2))\left[(\mathcal{I}^\gamma f)(x+k_1,y+k_2)-(\mathcal{I}^\gamma f)(x,y)\right]\\
=\int_0^1 \int_0^1 (x+k_1)^{\gamma_1}(1-u)^{\gamma_1-1}(y+k_2)^{\gamma_2}(1-v)^{\gamma_2-1}f \left((x+k_1)u,(y+k_2)v\right)dudv\\-\int_0^1 \int_0^1 x^{\gamma_1}(1-u)^{\gamma_1-1}y^{\gamma_2}(1-v)^{\gamma_2-1}f(xu,yv)dudv
=\int_0^1 \int_0^1 (x+k_1)^{\gamma_1}(1-u)^{\gamma_1-1}(y+k_2)^{\gamma_2}(1-v)^{\gamma_2-1}f\left((x+k_1)u,(y+k_2)v\right)dudv\\-\int_0^1 \int_0^1 (x+k_1)^{\gamma_1}(1-u)^{\gamma_1-1}(y+k_2)^{\gamma_2}(1-v)^{\gamma_2-1}f(xu,yv)dudv\\+\int_0^1 \int_0^1 (x+k_1)^{\gamma_1}(1-u)^{\gamma_1-1}(y+k_2)^{\gamma_2}(1-v)^{\gamma_2-1}f(xu,yv)dudv\\
-\int_0^1 \int_0^1 x^{\gamma_1}(1-u)^{\gamma_1-1}y^{\gamma_2}(1-v)^{\gamma_2-1}f(xu,yv)dudv\\
=\int_0^1 \int_0^1 (1-u)^{\gamma_1-1}(1-v)^{\gamma_2-1}(x+k_1)^{\gamma_1}(y+k_2)^{\gamma_2}\left[ f\left((x+k_1)u,(y+k_2)v\right)-f(xu,yv)\right] dudv\\
+\int_0^1 \int_0^1 (1-u)^{\gamma_1-1}(1-v)^{\gamma_2-1}f(xu,yv)\left[(x+k_1)^{\gamma_1}(y+k_2)^{\gamma_2}-x^{\gamma_1}y^{\gamma_2} \right]dudv.
\end{dmath*}
%%%%%%%%%%%%%%%%%%%%%%%%%%%%%%%%%%%%%
For $0<\delta <\frac{1}{2}, \frac{1}{\delta}<m,n<1+\frac{1}{\delta}$, let non negative integers $i$ and $j$ such that $0\le i \le m$,  $0\le j \le n$. Then 
 \begin{dmath*}
(\Gamma (\gamma_1)  \Gamma (\gamma_2))R_{\mathcal{I}^\gamma f}[A_{ij}]=\sup _{(x+k_1,y+k_2),(x,y)\in A_{ij}} \lvert (\mathcal{I}^\gamma f)(x+k_1,y+k_2)-(\mathcal{I}^\gamma f)(x,y)\rvert,
\end{dmath*}
where $A_{ij}=[i\delta,(i+1)\delta]\times [j\delta,(j+1)\delta].$\\
Here, 
\begin{dmath*}
=\lvert (\mathcal{I}^\gamma f)(x+k_1,y+k_2)-(\mathcal{I}^\gamma f)(x,y)\rvert
\leq (x+k_1)^{\gamma_1}(y+k_2)^{\gamma_2}\Bigl| \int_0^1 \int_0^1 (1-u)^{\gamma_1-1}(1-v)^{\gamma_2-1}\left[ f\left((x+k_1)u,(y+k_2)v\right)-f(xu,yv)\right] dudv\Bigr|\\
+\left((i+1)^{\gamma_1}(j+1)^{\gamma_2}-i^{\gamma_1} j^{\gamma_2})\right)\delta^{\gamma_1+\gamma_2}\int_0^1 \int_0^1 (1-u)^{\gamma_1-1}(1-v)^{\gamma_2-1}f(xu,yv)dudv.
\end{dmath*}
Let $i\geq 1,j\geq 1.$ On the one hand,
\begin{dmath*}
\Bigl| \int_0^1 \int_0^1 (1-u)^{\gamma_1-1}(1-v)^{\gamma_2-1}\left[ f\left((x+k_1)u,(y+k_2)v\right)-f(xu,yv)\right] dudv\Bigr|\\
=\Bigl| \int_0^{\frac{1}{i+1}} \int_0^{\frac{1}{j+1}} (1-u)^{\gamma_1-1}(1-v)^{\gamma_2-1}\left[ f\left((x+k_1)u,(y+k_2)v\right)-f(xu,yv)\right] dudv\Bigr|\\
+\sum_{p=1}^j \Bigl| \int_0^{\frac{1}{i+1}} \int_{\frac{p}{j+1}}^{\frac{p+1}{j+1}} (1-u)^{\gamma_1-1}(1-v)^{\gamma_2-1}\left[ f\left((x+k_1)u,(y+k_2)v\right)-f(xu,yv)\right] dudv\Bigr|\\
+\sum_{q=1}^i\Bigl| \int_{\frac{q}{i+1}}^{\frac{q+1}{i+1}} \int_0^{\frac{1}{j+1}} (1-u)^{\gamma_1-1}(1-v)^{\gamma_2-1}\left[ f\left((x+k_1)u,(y+k_2)v\right)-f(xu,yv)\right] dudv\Bigr|\\
+\sum_{q=1}^i \sum_{p=1}^j\Bigl| \int_{\frac{q}{i+1}}^{\frac{q+1}{i+1}} \int_{\frac{p}{j+1}}^{\frac{p+1}{j+1}} (1-u)^{\gamma_1-1}(1-v)^{\gamma_2-1}\left[ f\left((x+k_1)u,(y+k_2)v\right)-f(xu,yv)\right] dudv\Bigr|
\leq \frac{1}{(i+1)(j+1)}R_f\left[[0,\delta]\times [0,\delta]\right]\\
+\sum_{p=1}^j \frac{1}{(i+1)(j+1)}\left( R_f\left[[0,\delta]\times [(p-1)\delta,p\delta]\right]+R_f\left[[0,\delta]\times [p\delta,(p+l)\delta]\right]\right)\\
+\sum_{q=1}^i\frac{1}{(i+1)(j+1)}(R_f\left[[(q-1)\delta,q\delta]\times [0,\delta]\right]+R_f\left[[q\delta,(q+1)\delta]\times [0,\delta]\right])\\
+\sum_{q=1}^i \sum_{p=1}^j\frac{1}{(i+1)(j+1)}(R_f\left[[(q-1)\delta,q\delta]\times [(p-1)\delta,p\delta]\right]+R_f\left[[(q-1)\delta,q\delta]\times [p\delta,(p+1)\delta]\right]\\+R_f\left[[q\delta,(q+1)\delta]\times [(p-1)\delta,p\delta]\right]+R_f\left[[q\delta,(q+1)\delta]\times [p\delta,(p+1)\delta]\right]).
\end{dmath*}
By using Bernoulli's inequality $(1+u)^{r'}\leq 1+r'u$ for $0\leq r' \leq 1$ and $u\geq -1$, we can see that $$\int_0^{\frac{1}{i+1}} \int_0^{\frac{1}{j+1}} (1-u)^{\gamma_1-1}(1-v)^{\gamma_2-1}dudv \le \frac{1}{(i+1)(j+1)}. $$\\
On the other hand,
\begin{dmath*}
\left((i+1)^{\gamma_1}(j+1)^{\gamma_2}-i^{\gamma_1} j^{\gamma_2})\right)\delta^{\gamma_1+\gamma_2}\int_0^1 \int_0^1 (1-s)^{\gamma_1-1}(1-t)^{\gamma_2-1}f(xs,yt)dsdt\\
\leq (i+1)^{\gamma_1}(j+1)^{\gamma_2}\delta^{\gamma_1+\gamma_2} \frac{\max _{0\leq (x,y) \leq 1} f(x,y)}{\gamma_1\gamma_2}.
\end{dmath*}
From Lemma \ref{lmm}, we have
\begin{dmath*}
N_\delta(Gr(\mathcal{I}^\gamma f))\leq  2mn+\frac{1}{\delta} \sum_{j=1}^n \sum_{i=1}^m R_{\mathcal{I}^\gamma f}[A_{ij}] 
\leq 2mn+\frac{1}{\delta} \sum_{j=1}^n \sum_{i=1}^m \left(\frac{1}{(i+1)(j+1)}R_f\left[[0,\delta]\times [0,\delta]\right]\\ +\sum_{p=1}^j \frac{1}{(i+1)(j+1)}\left( R_f\left[[0,\delta]\times [(p-1)\delta,p\delta]\right]+R_f\left[[0,\delta]\times [p\delta,(p+l)\delta]\right]\right)\\+\sum_{q=1}^i\frac{1}{(i+1)(j+1)}(R_f\left[[(q-1)\delta,q\delta]\times [0,\delta]\right]+R_f\left[[q\delta,(q+1)\delta]\times [0,\delta]\right]) \\+\sum_{q=1}^i \sum_{p=1}^j\frac{1}{(i+1)(j+1)}(R_f\left[[(q-1)\delta,q\delta]\times [(p-1)\delta,p\delta]\right]+R_f\left[[(q-1)\delta,q\delta]\times \\ [p\delta,(p+1)\delta]\right]+R_f\left[[q\delta,(q+1)\delta]\times [(p-1)\delta,p\delta]\right]+R_f\left[[q\delta,(q+1)\delta]\times [p\delta,(p+1)\delta]\right])\right)\\
+\frac{1}{\delta} \sum_{j=1}^n \sum_{i=1}^m (i+1)^{\gamma_1}(j+1)^{\gamma_2}\delta^{\gamma_1+\gamma_2} \frac{\max _{0\leq (x,y) \leq 1} f(x,y)}{\gamma_1\gamma_2}
\leq \frac{1}{\delta}\left(C+ \sum_{j=1}^n \sum_{i=1}^m \left( \frac{1}{(i+1)(j+1)}R_f\left[[0,\delta]\times [0,\delta]\right]\\
+\sum_{p=1}^j \frac{1}{(i+1)(j+1)}\left( R_f\left[[0,\delta]\times [(p-1)\delta,p\delta]\right]+R_f\left[[0,\delta]\times [p\delta,(p+l)\delta]\right]\right)\\
+\sum_{q=1}^i\frac{1}{(i+1)(j+1)}(R_f\left[[(q-1)\delta,q\delta]\times [0,\delta]\right]+R_f\left[[q\delta,(q+1)\delta]\times [0,\delta]\right])\\
+\sum_{q=1}^i \sum_{p=1}^j\frac{1}{(i+1)(j+1)}(R_f\left[[(q-1)\delta,q\delta]\times [(p-1)\delta,p\delta]\right]+R_f\left[[(q-1)\delta,q\delta]\times [p\delta,(p+1)\delta]\right]+R_f\left[[q\delta,(q+1)\delta]\times [(p-1)\delta,p\delta]\right]+R_f\left[[q\delta,(q+1)\delta]\times [p\delta,(p+1)\delta]\right]) \right)\right)
\leq\frac{C}{\delta} \left( \sum_{j=0}^n \sum_{i=0}^m \frac{1}{(i+1)(j+1)} \right)\left( \sum_{j=1}^n \sum_{i=1}^m R_f[A_{ij}] \right)
\leq \frac{C}{\delta} (\log m)( \log n) \sum_{j=0}^n \sum_{i=0}^m R_f[A_{ij}] 
\leq C (\log m) (\log n) N_\delta(Gr(f)).
\end{dmath*}
Therefore, 
\begin{dmath*}
\frac{\log N_\delta(Gr(\mathcal{I}^\gamma f))}{-\log \delta}\leq \frac{\log \left \{ C (\log m) (\log n) N_\delta(Gr(f))\right\}}{-\log \delta}
\leq \frac{\log C}{-\log \delta}+\frac{\log(\log m)}{-\log \delta}+\frac{\log(\log n)}{-\log \delta}+\frac{\log N_\delta(Gr(f))}{-\log \delta}.
\end{dmath*}
So, we obtain
\begin{dmath*}
\overline{\dim}_B Gr(\mathcal{I}^{\gamma}f,[0,1]\times [0,1])=\overline{\lim_{\delta \to 0}}\frac{\log N_\delta(Gr(\mathcal{I}^\gamma f))}{-\log \delta}
\leq \overline{\lim_{\delta \to 0}} \left( \frac{\log C}{-\log \delta}+\frac{\log(\log m)}{-\log \delta}+\frac{\log(\log n)}{-\log \delta}+\frac{\log N_\delta(Gr(f))}{-\log \delta}\right)
\leq \overline{\lim_{\delta \to 0}}\frac{\log N_\delta(Gr(f))}{-\log \delta}
=\lim_{\delta \to 0}\frac{\log N_\delta(Gr(f))}{-\log \delta}=2.
\end{dmath*}
Thus Inequality \ref{14} holds. By combining Inequalities  \ref{12} and \ref{14}, we get desired result.
\end{proof}
\begin{corollary} Let $f:[0,1]\times[0,1]\to \mathbb{R}$ be a continuous function and $0<\gamma_1<1,~0<\gamma_2 <1.$ If $f$ is of bounded variation in Arzel\'{a} sense, then $$\dim_B Gr(\mathcal{I}^{\gamma}f,[0,1]\times[0,1])=2.$$
\end{corollary}
\begin{proof}
From Remark 3.13 in \cite{V1}, if $f:[0,1]\times[0,1]\to \mathbb{R}$ is continuous and of bounded variation in Arzel\'{a} sense on $[0,1]\times [0,1]$, then 
$$\dim_B Gr(f,[0,1]\times [0,1])=2.$$
Thus, by using Theorem \ref{Th5}, we get
$$\dim_B Gr(\mathcal{I}^\gamma f,[0,1]\times [0,1])=2.$$
This completes the proof.
\end{proof}
\begin{remark} Thus, Theorem 4.5 of \cite{V1} follows from our Theorem \ref{Th5}. In \cite{V1}, Verma and Viswanathan proved that the box dimension of the fractional integral of mixed R-L type of a continuous function which is of bounded variation in Arzel\'{a} sense on $[0,1]\times [0,1]$ is 2. Their results are more concern with analytical  aspects in the sense that they are using notion of bounded variation. But in the Theorem \ref{Th5}, we have proved that if a continuous function has box dimension two then the box dimension of its fractional integral of mixed R-L type is also two. That is, we are using the dimension of function to compute the dimension of its fractional integral of mixed R-L type. So our results are more concern with dimensional aspects.
\end{remark}
Now, we are going to corroborate the Theorem \ref{Th5} by using existing results.
\begin{lemma}
\cite{V1} \label{lmV2} Let a function $h:[c,d] \to \mathbb{R}$ be continuous. Consider a set as $H=\{(x,y,h(y)):x\in [a,b], y \in[c,d]\}$ with $a<b.$ Then it holds, $\overline{\dim}_B(H) \leq \overline{\dim}_B(Gr(h))+1.$
\end{lemma}
\begin{remark}\label{rmk1}
Let $h_1:[a,b] \to \mathbb{R}$ and $h_2:[c,d] \to \mathbb{R}$ are two continuous maps. Now, define $g_1,g_2:[a,b]\times [c,d]\to \mathbb{R}$ such that $$g_1(x,y)=h_1(x)+h_2(y), ~~\text{and}~~~g_2(x,y)=h_1(x)h_2(y).$$
From Lemma \ref{lmV2}, we get $\overline{\dim}_BGr(g_1)\le \overline{\dim}_BGr(h_2)+1$ and $\overline{\dim}_BGr(g_2)\le \overline{\dim}_BGr(h_2)+1.$
\end{remark}
\begin{remark}
Let $g:[a,b]\to \mathbb{R}$ be a continuous function which box dimension is $1$. We define a bivariate continuous function $f:[a,b]\times [c,d] \to \mathbb{R}$ such that $f(x,y)=g(x).$ From definition \ref{Def1}, we have
$$ \mathcal{I}^{\gamma}f(x,y)=\frac{1}{\Gamma (\gamma_1)  \Gamma (\gamma_2)} \int_a ^x \int_c ^y (x-u)^{\gamma_1-1} (y-v)^{\gamma_2-1}f(u,v)dudv.$$
For $\gamma_2=1,$ we get
$$ \mathcal{I}^{\gamma}f(x,y)=\frac{1}{\Gamma (\gamma_1)} \int_a ^x \int_c ^y (x-u)^{\gamma_1-1}f(u,v)dudv.$$
By definition of $f$, we obtain
$$ \mathcal{I}^{\gamma}f(x,y)=\frac{y-c}{\Gamma (\gamma_1)} \int_a ^x (x-u)^{\gamma_1-1}g(u)du.$$
So, we have a relation between the fractional integral of  R-L type of $g$, namely
$$\mathcal{I}^{\gamma_1}g(x)={\Gamma (\gamma_1)} \int_a ^x (x-u)^{\gamma_1-1}g(u)du,$$ 
and the fractional integral of mixed R-L type of $f$ as
$$\mathcal{I}^{\gamma}f(x,y)=(y-c)\mathcal{I}^{\gamma_1}g(x).$$
Now, from remark \ref{rmk1}, we know that $\overline{\dim}_BGr(\mathcal{I}^{\gamma}f)\le \overline{\dim}_BGr(\mathcal{I}^{\gamma_1}g)+1.$ Since, $\dim_B Gr(g)=1$, from Theorem 3.1 in \cite{L3}, it follows that $\dim_BGr(\mathcal{I}^{\gamma_1}g)=1,$ and hence $\dim_BGr(\mathcal{I}^{\gamma}f)=2.$ This corroborates the Theorem \ref{Th5}.
\end{remark}
%%%%%%%%%%%%%%%%%%%%%%%%%%%%%%%%%%%%%%%%%%%
\section{\textbf{Fractal Dimension  of $\mathcal{I}^{\gamma}f(x,y)$ of Unbounded Variational Continuous Functions}}
 First we give a sketch of construction of continuous functions having unbounded varaiational (UV) property at a single point. Then we investigate the fractal dimension of its  fractional integral of mixed R-L type. \\\\
\textbf{Construction of UV Continuous Functions:}\\
Consider $[0,1]\times [0,1].$ Let $\lim_{n \to \infty}a_n=1,$ where $(a_n)$ is the increasing sequence of real numbers in $[0,1].$ For our construction, we take a sequence $(a_n)_{n\ge 0}$ by considering $a_0=0$ and $a_n=\frac{1}{2}+\frac{1}{2^2}+...+\frac{1}{2^n},~ n\in \mathbb{N}.$ Let us define a continuous function $\Theta(x,y)=x(x-0.5)y$ on $[0,0.5]\times [0,1]$ such that $$\Theta(0,y)=\Theta(0.5,y)~~\forall~y\in [0,1].$$
We shall refer $\Theta$ as generating function. Let $\Upsilon_n$ be  map from $[a_{n-1},a_n]$ onto $[0,0.5]$ given by $$\Upsilon_n(x,y)=2^{(n-1)}(x-a_{n-1}).$$
Let us define $G_1(x,y)=\Theta(x,y)$ for $(x,y)\in [0,0.5]\times [0,1]$ and $n\ge 2,$ 
$$G_n(x,y)=\frac{1}{n}\Theta (\Upsilon_n(x),y)+\frac{n-1}{n}\Theta (0,y)~\text{for}~(x,y) \in [a_{n-1},a_n]\times [0,1].$$
Now, we denote that $F_n(x,y)$ is the composed of $G_1(x,y), G_2(x,y),...,G_n(x,y).$ Let $$M(x,y)=\lim_{n \to \infty} F_n(x,y).$$
The graph of $M(x,y)$ is given in Figure $1.$
\begin{lemma} The function $M$ is bounded and continuous on $[0,1]\times [0,1].$\end{lemma}
\begin{theorem}
The function $M$ is not of bounded variation on $[0,1]\times [0,1].$
\end{theorem}
\begin{proof}
It can be seen that $\Theta$ is non-constant function along the line $y=y_0$ for some $y_0\in [0,1].$ For $C>0$ and some $u_1,u_2 \in [a_0,a_1]$ with $u_1< u_2,$,  we have $$\lvert \Theta(u_1,y_0)-\Theta(u_2,y_0) \rvert \ge C.$$ Choose $w_1,w_2 \in [a_0,a_1], w_1 <w_2$ such that 
$$\lvert G_1(w_1,y_0)-G_1(w_2,y_0) \rvert=\lvert \Theta(t,y_0)-\Theta(u,y_0) \rvert\ge C.$$
We can choose $w_3,w_4\in [a_1,a_2], w_3<w_4$  such that
$$\lvert G_2(w_3,y_0)-G_2(w_4,y_0) \rvert=\frac{1}{2}\lvert \Theta(t,y_0)-\Theta(u,y_0) \rvert\ge \frac{C}{2}.$$
We can see that for $i>1, w_i=\Upsilon_{i-1}^{-1}(t)$ and $w_{i+1}=\Upsilon_{i-1}^{-1}(u).$ By proceeding in similar way, we get a collection
$P'=\{w_i:w_1<w_2<w_3<...<w_{2n}\}.$ Now, we take partition $P$ of $[0,1]$ such that $P'\subset P.$ The variation of $M$ along the line $y=y_0$ denoted by $V(M,[0,1],y_0)$ is $$V(M,[0,1],y_0)\ge \sum_{i=1}^{2n}\lvert M(w_{i+1},y_0)-M(w_i,y_0) \rvert \ge \sum_{i=1}^n \lvert G_i(w_{i+1},y_0)-G_i(w_i,y_0) \rvert  \ge \sum_{i=1}^n \frac{C}{i}.$$
Since $C>0$  and $ \sum_{i=1}^n \frac{1}{i}=\infty$, restriction of $M$, $M|_{y=y_0}$, is not of bounded variation on $[0,1]$ along the line $y=y_0$. So, $M|_{y=y_0}$ can not be written as difference of two increasing functions $g_1,g_2:[0,1]\to \mathbb{R}$ along the line $y=y_0.$ That is, $M|_{y=y_0}=g_1-g_2$ with $\Delta_{10}g_i(x,y_0)\ge 0,~i=1,2$ does not hold. Now, by using Theorem \ref{th2}, it is clear that the function $M$ is not of bounded variation on $[0,1]\times [0,1]$ in Arzel\'{a} sense.
\end{proof}
\begin{lemma}\cite{V1}\label{lm6.2}
 If $f(x,y)\in C([0,1]\times [0,1])$ and of bounded variation on $[0,1]\times [0,1]$ in Arzel\'{a} sense, then $\mathcal{I}^\gamma f(x,y)\in C([0,1]\times [0,1]) $ and of bounded variation on  $[0,1]\times [0,1]$ in Arzel\'{a} sense.
\end{lemma}
The following theorem gives the box dimension and the Hausdorff dimension of $\mathcal{I}^\gamma M(x,y).$
\begin{theorem} \label{Th6} Let $0<\gamma_1<1,~0<\gamma_2<1.$ Then
$\mathcal{I}^\gamma M(x,y)$ is finite on $[0,1]\times [0,1]$ and
$$\dim_H Gr(\mathcal{I}^\gamma M, [0,1]\times [0,1])=\dim_B Gr(\mathcal{I}^\gamma M, [0,1]\times [0,1])=2.$$
\end{theorem}
\begin{proof}
For $0<\gamma_1<1,~0<\gamma_2<1,$ we have 
\begin{dmath*}
\lvert \mathcal{I}^\gamma M(x,y) \rvert=\left \lvert \frac{1}{\Gamma (\gamma_1)  \Gamma (\gamma_2)} \int_0 ^x \int_0 ^y (x-u)^{\gamma_1-1} (y-v)^{\gamma_2-1}M(u,v)dudv \right \rvert 
\le \frac{1}{\Gamma (\gamma_1+1)\Gamma (\gamma_2+1)}x^{\gamma_1}y^{\gamma_2}\max_{(x,y)\in [0,1]\times [0,1]}\lvert M(x,y)\rvert
\le \frac{1}{\Gamma (\gamma_1+1)\Gamma (\gamma_2+1)}.
\end{dmath*}
\begin{figure}
\centering
\includegraphics[scale=0.5]{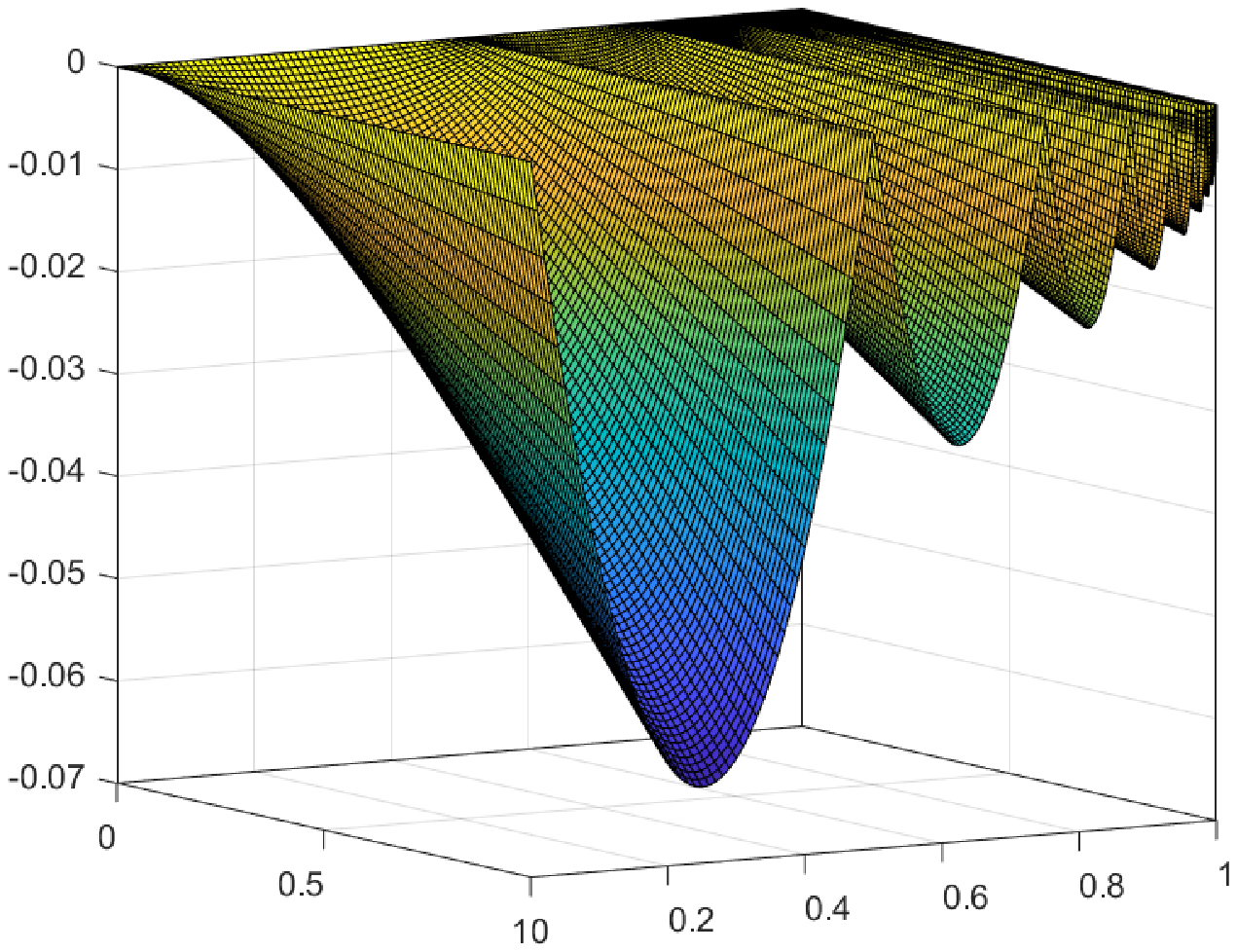}
\caption{ $M(x,y)$}
\end{figure}
This shows that $\mathcal{I}^\gamma M(x,y)$ is finite on $[0,1]\times [0,1]$. Since $M(x,y)$ is a continuous function and of bounded variation on $[0,1)\times [0,1),$ then from Lemma \ref{lm6.2} we know that $\mathcal{I}^\gamma M(x,y)$ is continuous and of bounded variation on  $[0,1)\times [0,1)$ for $0<\gamma_1<1,~0<\gamma_2<1$.\\
Let $0<\delta <1$ and a positive constant $C$, when $(x,y)\in [0,1-\delta)\times [0,1-\delta),$  $\mathcal{I}^\gamma M(x,y)$ is of bounded variation. Let the smallest number of sets of diameter $\delta$ which can cover graph of  $\mathcal{I}^\gamma M(x,y)$ is $\frac{C}{\delta^2}.$ Now, when $(x,y)\in [1,1-\delta)\times [1,1-\delta),$ then the number of $\delta$-cubes that intersect graph of $\mathcal{I}^\gamma M(x,y)$ is at most $\frac{1}{\delta^2}.$\\ Hence, the smallest number of sets of diameter $\delta$ which can cover graph of $\mathcal{I}^\gamma M(x,y)$ is at most $\frac{C+1}{\delta^2}.$  Thus, we have
\begin{dmath*}
\overline{\dim}_B Gr(\mathcal{I}^{\gamma}M,[0,1]\times [0,1])=\overline{\lim_{\delta \to 0}}\frac{\log N_\delta(Gr(\mathcal{I}^\gamma M))}{-\log \delta}
\le \lim_{\delta \to 0} \frac{\log \frac{C+1}{\delta^2}}{-\log \delta}
=2.
\end{dmath*}
From Definition \ref{DefB} and Lemma \ref{lmm}, we know that
\begin{equation*}
\underline{\dim}_B Gr(\mathcal{I}^{\gamma}M,[0,1]\times[0,1]) \geq 2.
\end{equation*}
This implies that 
\begin{equation}\label{6.1}
\dim_B Gr(\mathcal{I}^{\gamma}M,[0,1]\times[0,1])= 2.
\end{equation}
Also, we know that
\begin{equation}\label{6.2}
 2 \le \dim_H Gr(\mathcal{I}^{\gamma}M,[0,1]\times[0,1]) \le \dim_B Gr(\mathcal{I}^{\gamma}M,[0,1]\times[0,1]).
\end{equation}
From Equations \ref{6.1} and \ref{6.2}, we get the required result.
 \end{proof}
 \begin{remark}
 From \cite{V1}, we know that if a function is continuous and of bounded variation in Arzel\'{a} sense, then its fractional integral of mixed R-L type is also continuous and of bounded variation in Arzel\'{a} sense and its fractal dimension is 2. From Theorem \ref{Th6}, we conclude that the box dimension and the Hausdorff dimension of the fractional integral of mixed R-L type of unbounded variational continuous function are also 2. So, $M$ is the such example which is of unbounded variational continuous function but the fractal dimension of its the fractional integral of mixed R-L type is 2. 
 \end{remark}
%%%%%%%%%%%%%%%%%%%%%%%%%%%%%%%%%%%%%%%%%
%%%%%%%%%%%%%%%%%%%%%%%%%%%%%%%%%%%%%%%%%%%%%%%%%%%%%%%%%%%%%%%
%\section{\textbf{Open Problems}}
%\begin{itemize}
%\item[(i)]FD of  Fractal function
%\item[(ii)] If box dim is greater than $2$..
%\end{itemize}

%%%%%%%%%%%%%%%%%%%%%%%%%%%%%%%%%%%%%%%%%%%%%%%%%%%
\subsection*{\textbf{Acknowledgements}} First author has received the  financial supported from the CSIR, India (file no: 09/1058(0012)/2018-EMR-I).
%%%%%%%%%%%%%%%%%%%%%%%%%%%%%%%%%%%%%%%%%%%%%%%%%%%%%%%%%%%%%%
\bibliographystyle{amsplain}

\end{document}